\documentclass[twoside,american]{amsart}
\usepackage{mathptmx}
\usepackage[T1]{fontenc}
\usepackage[latin9]{inputenc}
\usepackage[a4paper]{geometry}
\usepackage{babel}

\usepackage{prettyref}
\usepackage{units}
\usepackage{amsmath}
\usepackage{amssymb}
\usepackage[unicode=true, pdfusetitle,
 bookmarks=true,bookmarksnumbered=false,bookmarksopen=false,
 breaklinks=false,pdfborder={0 0 1},backref=false,colorlinks=false]
 {hyperref}
\hypersetup{
 bookmarksopen}
 
\makeatletter
\theoremstyle{plain}
\newtheorem{stthm}{Theorem}[section]

\numberwithin{equation}{section}
\numberwithin{figure}{section}
\numberwithin{table}{section}
\def\newrefformat#1#2{%
  \@namedef{pr@#1}##1{#2}}
\newrefformat{eq}{\textup{(\ref{#1})}}
\newrefformat{sec}{Section \ref{#1}}
\def\prettyref#1{\@prettyref#1|}
\def\@prettyref#1|#2|{%
  \expandafter\ifx\csname pr@#1\endcsname\relax%
    \PackageWarning{prettyref}{Reference format #1\space undefined}%
    \ref{#1|#2}%
  \else%
    \csname pr@#1\endcsname{#1|#2}%
  \fi%
}
  \def\AS{\\ }
  \def\SU{\substack}
\def\Compatamsartbabel#1#2{%
\def\rightmark{#1}%
\def\leftmark{#2}%
}
 \theoremstyle{definition}
 \newtheorem{sddefn}{Definition}[section]
 \theoremstyle{remark}
 \newtheorem*{rem*}{Remark}
 \theoremstyle{plain}
 \newtheorem{sllem}{Lemma}[section]
 \theoremstyle{plain}
 \newtheorem*{lem*}{Lemma}

\newrefformat{lem}{Lemma \ref{#1}}
\newrefformat{thm}{Theorem \ref{#1}}
\newrefformat{cor}{Corollary \ref{#1}}

\makeatother

\begin{document}

\Compatamsartbabel{SOLUTIONS OF THE $\bar\partial$-EQUATION FOR LINEALLY CONVEX DOMAINS AND ZEROS SETS OF THE NEVANLINNA CLASS}{P. CHARPENTIER, Y. DUPAIN \& M. MOUNKAILA}

\title{Estimates for Solutions of the $\bar{\partial}$-equation and Application
to the Characterization of the zero varieties of the Functions of
the Nevanlinna class for lineally convex domains of Finite Type}

\author{Philippe Charpentier, Yves Dupain \& Modi Mounkaila}
\begin{abstract}
In the late ten years, the resolution of the equation $\bar{\partial}u=f$
with sharp estimates has been intensively studied for convex domains
of finite type in $\mathbb{C}^{n}$ by many authors. Generally they
used kernels constructed with holomorphic support function satisfying
{}``good'' global estimates. In this paper, we consider the case
of lineally convex domains. Unfortunately, the method used to obtain
global estimates for the support function cannot be carried out in
that case. Then we use a kernel that does not gives directly a solution
of the $\bar{\partial}$-equation but only a representation formula
which allows us to end the resolution of the equation using Kohn's
$L^{2}$ theory.

As an application we give the characterization of the zero sets of
the functions of the Nevanlinna class for lineally convex domains of finite
type.
\end{abstract}

\keywords{lineally convex, finite type, $\bar{\partial}$-equation, Nevanlinna
class}

\subjclass[2000]{32F17, 32T25, 32T40}

\address{P. Charpentier \& Y. Dupain, Université Bordeaux I, Institut de Mathématiques
de Bordeaux, 351, Cours de la Libération, 33405, Talence, France}

\address{M. Mounkaila, Université Abdou Moumouni, Faculté des Sciences, B.P.
10662, Niamey, Niger}

\email{P. Charpentier: philippe.charpentier@math.u-bordeaux1.fr}

\email{Y. Dupain: yves.dupain@math.u-bordeaux1.fr}

\email{M. Mounkaila: modi.mounkaila@yahoo.fr}

\maketitle

\section{\label{sec|Introduction}Introduction}

The general notion of extremal basis and the class of {}``geometrically
separated'' domains has been introduced in \cite{CD08}. For such
domains it is proved that if there exist {}``good'' plurisubharmonic
functions, in which case the domains are called {}``completely geometrically
separated'', then sharp estimates on the Bergman and Szegö projections
and on the classical invariants metrics can be obtained.

Moreover, using the description of the complex geometry of lineally
convex domains of finite type initiated in \cite{Conrad_lineally_convex}
and the construction of a local support function described in \cite{Diederich-Fornaess-Support-Func-lineally-cvx},
it is shown, already in \cite{CD08}, that every lineally convex domain
of finite type is completely geometrically separated.

The present paper is a continuation of the study of complex analysis
in such domains. We are now interested in the problem of the characterization
of the zero sets of functions in the Nevanlinna class. The main result
obtained concerns the class of lineally convex domains of finite type
and generalizes the results obtained in the case of convex domains
(\cite{Bruna-Charp-Dupain-Annals,Cumenge-Navanlinna-convex,DM01}):
\begin{stthm}
\label{thm|zeros-Nevalinna}Let $\Omega$ be a bounded lineally convex
domain of finite type in $\mathbb{C}^{n}$ with smooth boundary. Then
a divisor in $\Omega$ can be defined by a function of the Nevanlinna
class of $\Omega$ if and only if it satisfies the Blaschke condition.
\end{stthm}
The general scheme of the proof is identical to the one used in the
convex case and consists in three steps. First, for the general case
of geometrically separated domains, we prove some {}``Malliavin conditions''
on closed positive $\left(1,1\right)$-currents $\Theta$ and then
we solve the equation $dw=\Theta$ with good estimates. The third
step, which solves the $\bar{\partial}$-equation for $\left(0,1\right)$-form with
$L^{1}$ estimates on the boundary, is only done in the case of lineally convex domains
of finite type:
\begin{stthm}
\label{thm|general-d-bar-resolution}Let $\Omega$ be a bounded lineally
convex domain of finite type in $\mathbb{C}^{n}$ with smooth boundary.
Let $f$ be a $\left(0,1\right)$-form in $\Omega$ whose coefficients are
$\mathcal{C}^1(\overline\Omega)$ functions and which is $\bar{\partial}$-closed. Then there
exists a solution of the equation $\overline\partial u=f$, smooth on $\Omega$ and continuous
on $\overline\Omega$ such that $\left\Vert u\right\Vert_{L^1(\partial\Omega)}\leq C|\!|\!|f|\!|\!|_{k}$
(see \prettyref{sec|Geometry-and-local-support-function} formula \prettyref{eq|k-morm-measure}), the constant $C$ depending only on $\Omega$. In other words, there exists a
solution of the equation $\overline{\partial_{b}}u=f$, in the sense of \cite{Sko76},
in $L^{1}(\partial\Omega)$.
\end{stthm}

\section{\label{sec|resolution-d-bar-equation}Solutions for the $\bar{\partial}$-equation
for lineally convex domains of finite type}

First of all, we recall the definition of lineally convex domain:
\begin{sddefn}
A domain $\Omega$ in $\mathbb{C}^{n}$, with smooth boundary is said
to be lineally convex at a point $p\in\partial\Omega$ if there exists
a neighborhood $W$ of $p$ such that, for all point $z\in\partial\Omega\cap W$,\[
\left(z+T_{z}^{10}\right)\cap(D\cap W)=\emptyset,\]
where $T_{z}^{10}$ is the holomorphic tangent space to $\partial\Omega$
at the point $z$.
\end{sddefn}
Furthermore, we always suppose that $\partial\Omega$ is of finite
type at every point of $\partial\Omega\cap W$. Shrinking $W$ if
necessary, we may assume that there exists a $\mathcal{C}^{\infty}$defining
function $\rho$ for $\Omega$ and a number $\eta_{0}>0$ such that
$\nabla\rho(z)\neq0$ at every point of $W$ and the level sets $\left\{ z\in W\mbox{ such that }\rho(z)=\eta\right\} $,
$-\eta_{0}\leq\eta\leq\eta_{0}$, are lineally convex of finite type.

As we want to obtain global results, we need these properties at every
boundary point. Thus, in all our work, by {}``lineally convex domain''
we mean a bounded smooth domain having a (global) defining function
satisfying the previous hypothesis at every point of $\partial\Omega$.

\medskip{}

In \prettyref{sec|Geometry-and-local-support-function} we define
a punctual anisotropic norm for forms, $\left\Vert .\right\Vert _{k}$,
related to the geometry of the domain (formula \prettyref{eq|anisotropic_distance_for_0-q_forms}).
With this notation, the main goal of this Section is to prove the
following reformulation of \prettyref{thm|general-d-bar-resolution} for
$\left(0,q\right)$-forms:
\begin{stthm}
\label{thm|L1-estimate_boundary_d-bar}Let $\Omega$ be a smooth bounded
lineally convex domain of finite type in $\mathbb{C}^{n}$. Then there
exists a constant $C>0$ such that, for any smooth $\bar{\partial}$-closed
$\left(0,q\right)$-form $f$, $1\leq q\leq n-1$, on $\overline{\Omega}$
there exists a solution $u$ of the equation $\overline{\partial}u=f$,
continuous on $\overline{\Omega}$, such that
\[\int_{\partial\Omega}\left|u(z)\right|d\sigma(z)\leq C\int_{\Omega}\left\Vert f(z)\right\Vert _{k}dV(z).\]

\end{stthm}

Except for the case of finite type domains in $\mathbb{C}^{2}$ where
such an estimate was proved by D. C. Chang, A. Nagel and E. M. Stein
(\cite{Chang-Nagel-Stein}) for the $\bar{\partial}$-Neumann problem,
this kind of result was always proved using explicit kernels solving
the $\bar{\partial}$-equation. The first result was obtained, independently,
by G. M. Henkin and H. Skoda for strictly pseudo-convex domains (\cite{Hen75,Sko76}).
Afterward, some generalizations to special pseudo-convex domains of
finite type (in dimension $n\geq3$) were obtained by several authors.
For example, the case of complex ellipsoids was obtained by A. Bonami
and Ph. Charpentier (\cite{BC82}), and, probably the most notable
result, the case of convex domains of finite type by J. Bruna, Ph.
Charpentier \& Y. Dupain, A. Cumenge and K. Diederich \& E. Mazzilli
(\cite{Bruna-Charp-Dupain-Annals,Cumenge-Navanlinna-convex,DM01}).

Here, we consider the more general case of lineally convex domains
of finite type. Our starting point is similar to the one used in \cite{Cumenge-Navanlinna-convex}
and \cite{DM01}. We try to construct a kernel solving the $\bar{\partial}$-equation
following the method described in the classical paper of B. Berndtsson
and M. Andersson \cite{BA82}. Such kernel is constructed using two
forms, $s(z,\zeta)$ and $Q(z,\zeta)$ satisfying some conditions.
In particular $Q(z,\zeta)$ is supposed to be holomorphic in $z$.
In many works using these constructions, the forms $s$ and $Q$ (or
only $Q$) are constructed using a holomorphic support function for
the domain. In the case of lineally convex domains of finite type
such support functions have been constructed by K. Diederich and J.
E. Fornaess in \cite{Diederich-Fornaess-Support-Func-lineally-cvx}.
Let us denote by $S_{0}=\sum Q_{i}^{0}(z,\zeta)\left(z_{i}-\zeta_{i}\right)$
this support function. If we want to use $S_{0}$ to define $s$ and/or
$Q$, a problem appears immediately: some precise global estimates
of $S_{0}$ are necessary since this function appears in the denominators
of the kernels, and Diederich and Fornaess result gives only local
estimates (i.e. when the two points $z$ and $\zeta$ are close and
close to the boundary of the domain). This problem has been noticed
previously in the case of convex domains by W. Alexandre in \cite{Ale01}
where a modification of the support function is done. Unfortunately,
this modification cannot apply for lineally convex domains, the convexity
being strongly used to solve a division problem with estimates. Another
way to construct a kernel without support function, introduced by
A. Cumenge for convex domains in \cite{Cumenge-Navanlinna-convex},
is to use the Bergman kernel with the estimates obtained in \cite{McNeal-convexes-94}.
These needed estimates on the Bergman kernel have been obtained for
lineally convex domains in \cite{CD08} but, once again, the method
cannot be carried out for lineally convex domains for the same reason.

Thus we start with the method of Berndtsson and Andersson with $Q$
constructed with $S_{0}$, but the form $Q$ being holomorphic in
$z$ only when the two points $z$ and $\zeta$ are close (and close
to the boundary). Then the construction does not give a kernel solving
the $\bar{\partial}$-equation but a representation formula of the
following form: if $f$ is a $\left(0,q\right)$-form smooth in $\bar{\Omega}$,
there exist kernels $K(z,\zeta)$, $K_{1}(z,\zeta)$ and $P(z,\zeta)$
such that\[
f(z)=\bar{\partial}\left(\int_{\Omega}f(\zeta)\wedge K(z,\zeta)\right)+\int_{\Omega}\bar{\partial}f(\zeta)\wedge K_{1}(z,\zeta)+\int_{\Omega}f(\zeta)\wedge P(z,\zeta).\]
In this formula one important point is that, by construction, the
kernel $P$ is $\mathcal{C}^{\infty}\left(\bar{\Omega}\times\bar{\Omega}\right)$.
If $f$ is $\bar{\partial}$-closed then the form $g=\int_{\Omega}f(\zeta)\wedge P(z,\zeta)$
is also $\bar{\partial}$-closed, and, by the regularity of $P$ (\prettyref{lem|P-C-infinity}),
for all integer $r$, the Sobolev norm $\left\Vert g\right\Vert _{W^{r}}$
of order $r$ is controlled by $C_{r}\left\Vert f\right\Vert _{L^{1}(\Omega)}$.
Then, using Kohn's theory (\cite{Kohn-defining-function}), it is possible,
using a linear operator, to solve the equation $\bar{\partial}v=g$ with an estimate
of the form $\left\Vert v\right\Vert _{W^{r}}\leq C_{r}\left\Vert f\right\Vert _{L^{1}(\Omega)}$ (\prettyref{lem|Sob-Est-P}).
By Sobolev Lemma, choosing $r$ sufficiently large (depending only on the dimension),
there exists a constant $C$ such that, if $f\in L^{1}(\Omega)$, this solution
$v$ is continuous on $\overline{\Omega}$ and $\left\Vert v \right\Vert _{L^{1}(\partial\Omega)}\leq C\left\Vert f\right\Vert _{L^{1}(\Omega)}$.  
Finally, to obtain a solution of the equation $\bar{\partial}u=f$, given by a linear operator,
satisfying the desired estimate it suffices to estimate the integral $\int_{\Omega}f(\zeta)\wedge K(z,\zeta)$
which can be done, as we will see, using only the local estimates
of the support function $S_{0}$ given in \cite{Diederich-Fornaess-Support-Func-lineally-cvx}.

\subsection{Geometry and local support function\label{sec|Geometry-and-local-support-function}}

\subsubsection{Geometry of lineally convex domains of finite type\label{sec|Geometry-of-lineally-convex}}

Adapting the construction made by J. McNeal for convex domains (\cite{McNeal-convexes-94})
to the case of lineally convex domains of finite type, M. Conrad defined,
in \cite{Conrad_lineally_convex}, the geometry of these domains and,
in particular, the notion of extremal basis in this context (note
that in his construction the basis are not maximal but minimal, see
\cite{Hef04,NPT09} for more details). Here we will only recall the
results which are useful for our purpose. A more detailed exposition
is given in \cite{Diederich-Fischer_Holder-linally-convex}.

For $\zeta$ close to $\partial\Omega$ and $\varepsilon\leq\varepsilon_{0}$,
$\varepsilon_{0}$ small, define, for all unitary vector $v$,\[
\tau\left(\zeta,v,\varepsilon\right)=\sup\left\{ c\mbox{ such that }\left|\rho\left(\zeta+\lambda v\right)-\rho(\zeta)\right|<\varepsilon,\,\forall\lambda\in\mathbb{C},\,\left|\lambda\right|<c\right\} .\]
Note that, if $v$ is tangent to the level set of $\rho$ passing
through $\zeta$, $\tau\left(\zeta,v,\varepsilon\right)\gtrsim\varepsilon^{\nicefrac{1}{2}}$
(with uniform constant in $\zeta$, $v$ and $\varepsilon$) and that,
$\Omega$ being of finite type $\leq2m$, $\tau\left(\zeta,v,\varepsilon\right)\lesssim\varepsilon^{\nicefrac{1}{2m}}$.

Let $\zeta$ and $\varepsilon$ be fixed. Then, an orthonormal basis
$\left(v_{1},v_{2},\ldots,v_{n}\right)$ is called \emph{$\left(\zeta,\varepsilon\right)$-extremal}
(or $\varepsilon$-\emph{extremal}, or simply \emph{extremal}) if
$v_{1}$ is the complex normal (to $\rho$) at $\zeta$, and, for
$i>1$, $v_{i}$ belongs to the orthogonal space of the vector space
generated by $\left(v_{1},\ldots,v_{i-1}\right)$ and minimizes $\tau\left(\zeta,v,\varepsilon\right)$
in that space. In association to this extremal basis, we denote\[
\tau(\zeta,v_{i},\varepsilon)=\tau_{i}(\zeta,\varepsilon).\]

Note that there may exist many \emph{$\left(\zeta,\varepsilon\right)$}-extremal
bases but they all give the same geometry we recall now.

With these notations, one defines polydiscs $AP_{\varepsilon}(\zeta)$
by\[
AP_{\varepsilon}(\zeta)=\left\{ z=\zeta+\sum_{k=1}^{n}\lambda_{k}v_{k}\mbox{ such that }\left|\lambda_{k}\right|\leq c_{0}A\tau_{k}(\zeta,\varepsilon)\right\} ,\]
$c_{0}$ depending on $\Omega$, $P_{\varepsilon}(\zeta)$ being the
corresponding polydisc with $A=1$ and we also define\[
d(\zeta,z)=\inf\left\{ \varepsilon\mbox{ such that }z\in P_{\varepsilon}(\zeta)\right\} .\]
The fundamental result here is that $d$ is a pseudo-distance which
means that, $\forall\alpha>0$, there exist constants $c(\alpha)$
and $C(\alpha)$ such that \begin{equation}
c(\alpha)P_{\varepsilon}(\zeta)\subset P_{\alpha\varepsilon}(\zeta)\subset C(\alpha)P_{\varepsilon}(\zeta)\mbox{ and }P_{c(\alpha)\varepsilon}(\zeta)\subset\alpha P_{\varepsilon}(\zeta)\subset P_{C(\alpha)\varepsilon}(\zeta).\label{eq|polydiscs-pseudodistance}\end{equation}
We insist on the fact that this pseudodistance is well defined and
is independent of the choice of the extremal bases.

We will make use of the following properties:
\begin{enumerate}
\item \label{geometry-1}Let $w=\left(w_{1},\ldots,w_{n}\right)$ be an
orthonormal system of coordinates centered at $\zeta$. Then\[
\left|\frac{\partial^{\left|\alpha+\beta\right|}\rho(\zeta)}{\partial w^{\alpha}\partial\bar{w}^{\beta}}\right|\lesssim\frac{\varepsilon}{\prod_{i}\tau\left(\zeta,w_{i},\varepsilon\right)^{\alpha_{i}+\beta_{i}}},\,\left|\alpha+\beta\right|\geq1.\]

\item \label{geometry-2}Let $\nu$ be a unit vector. Let $a_{\alpha\beta}^{\nu}(\zeta)=\frac{\partial^{\alpha+\beta}\rho}{\partial\lambda^{\alpha}\partial\bar{\lambda}^{\beta}}\left(\zeta+\lambda\nu\right)_{|\lambda=0}$.
Then\[
\sum_{1\leq\left|\alpha+\beta\right|\leq2m}\left|a_{\alpha\beta}^{\nu}(\zeta)\right|\tau(\zeta,\nu,\varepsilon)^{\alpha+\beta}\simeq\varepsilon.\]

\item \label{geometry-3}If $\left(v_{1},\ldots,v_{n}\right)$ is a $\left(\zeta,\varepsilon\right)$-extremal
basis and $\gamma=\sum_{1}^{n}a_{j}v_{j}\neq0$, then\[
\frac{1}{\tau(\zeta,\gamma,\varepsilon)}\simeq\sum_{j=1}^{n}\frac{\left|a_{j}\right|}{\tau_{j}(\zeta,\varepsilon)}.\]

\item \label{geometry-4}If $v$ is a unit vector then:

\begin{enumerate}
\item $z=\zeta+\lambda v\in P_{\delta}(\zeta)$ implies $\left|\lambda\right|\lesssim\tau(\zeta,v,\delta)$,
\item $z=\zeta+\lambda v$ with $\left|\lambda\right|\leq\tau(\zeta,v,\delta)$
implies $z\in CP_{\delta}(\zeta)$.
\end{enumerate}
\item \label{geometry-5}$\tau_{1}(\zeta,\varepsilon)=\varepsilon$, and,
for $j>1$ and $\lambda\geq1$, $\lambda^{\nicefrac{1}{m}}\tau_{j}(\zeta,\varepsilon)\lesssim\tau_{j}(\zeta,\lambda\varepsilon)\lesssim\lambda^{\nicefrac{1}{2}}\tau_{j}(\zeta,\varepsilon)$.\end{enumerate}
\begin{rem*}
Every lineally convex domain of finite type is completely geometrically
separated and the pseudo-distance defined here is equivalent to the
one defined in \cite{CD08} using tangent complex vector fields (see
Section 7.1 of \cite{CD08} for some details).

\medskip{}

With these notations, we define a punctual anisotropic norm $\left\Vert .\right\Vert _{k}$
for $\left(0,q\right)$-forms with functions coefficients $f$ by\begin{equation}
\left\Vert f(z)\right\Vert _{k}=\sup_{\left\Vert v_{i}\right\Vert =1}\frac{\left|\left\langle f;v_{1},\ldots,v_{q}\right\rangle (z)\right|}{\sum_{i=1}^{q}k\left(z,v_{i}\right)},\label{eq|anisotropic_distance_for_0-q_forms}\end{equation}
where $k\left(z,v\right)=\frac{\delta_{\partial\Omega}(z)}{\tau\left(z,v,\delta_{\partial\Omega}(z)\right)}$,
$\delta_{\partial\Omega}(z)$ being the distance of $z$ to the boundary
of $\Omega$. Note that this definition generalizes the definition
given in \cite{CD08} for $\left(0,1\right)$-forms. Moreover, in
the coordinate system associated to an $\left(z,\delta(z)\right)$-extremal
basis, we have $\left\Vert d\bar{z}^{I}\right\Vert _{k}\simeq\min_{i\in I}\frac{\tau_{i}\left(z,\delta_{\partial\Omega}(z)\right)}{\delta_{\partial\Omega}(z)}$,
and, if $f=\sum_{I}a_{I}d\bar{z}^{I}$, \[
\left\Vert f\right\Vert _{k}\simeq\sup_{I}\left|a_{I}\right|\min_{i\in I}\frac{\tau_{i}\left(z,\delta_{\partial\Omega}(z)\right)}{\delta_{\partial\Omega}(z)}.\]

If $f$ is a $\left(0,q\right)$-form with continuous coefficients, $\left\Vert f(z)\right\Vert_k$ is also continuous and we
define it's $|\!|\!|.|\!|\!|_{k}$ norm by\begin{equation}
|\!|\!|f|\!|\!|_{k}=\int_{\Omega}\left\Vert f(z)\right\Vert_kdV(z).\label{eq|k-morm-measure}\end{equation}
\end{rem*}

\subsubsection{The holomorphic support function\label{sec|The-holomorphic-support-function}}

In \cite{Diederich-Fornaess-Support-Func-lineally-cvx} the following
result is proved:
\begin{stthm}
[K. Diederich \& J. E. Fornaess]\label{thm|Die-For_support_function}Let
$\Omega$ be a bounded lineally convex domain in $\mathbb{C}^{n}$
of finite type $2m$ with $\mathcal{C}^{\infty}$ boundary. Then there
exist a neighborhood $W_{0}$ of the boundary of $\Omega$ and, for
any $\varepsilon>0$ small enough a function $S_{0}(z,\zeta)\in\mathcal{C}^{\infty}\left(\mathbb{C}^{n},W_{0}\right)$
which is a holomorphic polynomial of degree $2m$ in $z$ for any
$\zeta\in W_{0}$ fixed, such that $S_{0}(\zeta,\zeta)=0$, satisfying
the following precise properties:

Let $M$, $K>0$ be chosen sufficiently large and $\varepsilon>0$
sufficiently small. Choose $l_{\zeta}$ a family of affine unitary
transformations on $W_{0}$ translating $\zeta$ to $0$ and rotating
the complex normal $n_{\zeta}$ to $\rho$ at $\zeta$ to the vector
$\left(1,0,\ldots,0\right)$. Then there exists, on $W_{0}$, a family
of holomorphic polynomials $A_{\zeta}$, $A_{\zeta}(0)=0$, such that,
if $\Phi_{\zeta}$ is defined by $\Phi_{\zeta}^{-1}(z)_{1}=z_{1}\left(1-A_{\zeta}(z)\right)$,
$\Phi_{\zeta}^{-1}(z)_{k}=z_{k}$, $k=2,\ldots,n$, then\begin{equation}
S_{0}\left(l_{\zeta}\circ\Phi_{\zeta}(\xi),\zeta\right)=\xi_{1}+K\xi_{1}^{2}-\varepsilon\sum_{j=2}^{2m}M^{2^{j}}\sigma_{j}\sum_{\SU{\left|\alpha\right|=j\AS\alpha=\left(0,\alpha_{1},\ldots,\alpha_{n}\right)}}\frac{1}{\alpha!}\frac{\partial\rho_{\zeta}(0)}{\partial\xi^{\alpha}}\xi^{\alpha}\label{eq|definition-formula-supp-func}\end{equation}
where $\rho_{\zeta}(\xi)=\rho\left(l_{\zeta}\circ\Phi_{\zeta}(\xi)\right)-\rho(\zeta)$
and\[
\sigma_{j}=\left\{ \begin{array}{ll}
1 & \mbox{for }j=0\mbox{ mod }4\\
-1 & \mbox{for }j=2\mbox{ mod }4\\
0 & \mbox{otherwise}\end{array}\right..\]

Moreover, there exist $d=d(\varepsilon)>0$ and $c>0$ such that,
if $n_{\zeta}$ is the unit real exterior normal to $\rho$ at $\zeta$,
for $\left(w_{1},w_{2}\right)\in\mathbb{C}^{2}$ and $t$ a unit vector
in the holomorphic tangent space to $\rho$ at $\zeta$, for $\left|w\right|<d$,
the following estimate holds\begin{equation}
\Re\mathrm{e}S_{0}\left(\zeta+w_{1}n_{\zeta}+w_{2}t,\zeta\right)\leq\left[\rho\left(\zeta+w_{1}n_{\zeta}+w_{2}t\right)-\rho(\zeta)\right]h\left(\zeta+w_{1}n_{\zeta}+w_{2}t\right)-\varepsilon c\sum_{j=2}^{n}\left\Vert P_{\zeta,t}^{j}\right\Vert \left|w_{2}\right|^{j},\label{eq|local-estimate-supp-func}\end{equation}
where $h$ is a positive function bounded away from $0$, $P_{\zeta,t}^{j}(w)=P_{\zeta}^{j}\left(\zeta+w_{1}n_{\zeta}+w_{2}t\right)$,
with\[
P_{\zeta}^{j}(z)=\sum_{\left|\alpha\right|+\left|\beta\right|=j}\frac{1}{\alpha!\beta!}\frac{\partial^{j}\rho(\zeta)}{\partial z^{\alpha}\partial\bar{z}^{\beta}}\left(z-\zeta\right)^{\alpha}\left(\bar{z}-\bar{\zeta}\right)^{\beta}\]
and, for any polynomial $P=\sum a_{\alpha\beta}z^{\alpha}\bar{z}^{\beta}$,
$\left\Vert P\right\Vert =\sum\left|a_{\alpha\beta}\right|$.\end{stthm}
\begin{rem*}
In the above Theorem the function $S_{0}$ is globally defined for
$\zeta\in W_{0}$, and \prettyref{eq|definition-formula-supp-func}
is independent of the choice of $l_{\zeta}$. In particular, as it
is stated in \cite{Diederich-Fornaess-Support-Func-lineally-cvx},
if we restrict $\zeta$ to a small open set in $W_{0}$, the functions
$l_{\zeta}$, $h$ and $A_{\zeta}$ can be chosen $\mathcal{C}^{\infty}$
in that set (with respect to the two variables $\zeta$ and $z$).
\end{rem*}

\subsection{Koppelman formulas\label{sec|Koppelman-formulas}}

With the notations used for the holomorphic support function $S_{0}$,
we choose $R<d$ such that $\left|A_{\zeta}(z)\right|<10^{-1}$ if
$\left|z-\zeta\right|<R$ and, reducing $\eta_{0}$ if necessary,
we may suppose that $\delta_{\partial\Omega}(\zeta)<\eta_{0}$ implies
$\zeta\in W_{0}$.

Let us define two $\mathcal{C}^{\infty}$ functions $\chi_{1}(z,\zeta)=\hat{\chi}\left(\left|z-\zeta\right|\right)$
and $\chi_{2}(z)=\tilde{\chi}\left(\delta_{\partial\Omega}(z)\right)$
(where $\delta_{\partial\Omega}$ denotes the distance to the boundary
of $\Omega$) where $\chi$ and $\tilde{\chi}$ are $\mathcal{C}^{\infty}$
functions, $0\leq\hat{\chi},\,\tilde{\chi}\leq1$, such that $\hat{\chi}\equiv1$
on $\left[0,\nicefrac{R}{2}\right]$ and $\hat{\chi}\equiv0$ on $\left[R,+\infty\right[$
and $\tilde{\chi}\equiv1$ on $\left[0,\nicefrac{\eta_{0}}{2}\right]$
and $\tilde{\chi}\equiv0$ on $\left[\eta_{0},+\infty\right[$. Then
we define \[
\chi(z,\zeta)=\chi_{1}(z,\zeta)\chi_{2}(\zeta)\]
and\[
S(z,\zeta)=\chi(z,\zeta)S_{0}(z,\zeta)-\left(1-\chi(z,\zeta)\right)\left|z-\zeta\right|^{2}=\sum_{i=1}^{n}Q_{i}(z,\zeta)\left(z_{i}-\zeta_{i}\right).\]

Now we define the two forms $s$ and $Q$ used in \cite{BA82} in
the construction of the Koppelman formula by\[
s(z,\zeta)=\sum_{i=1}^{n}\left(\overline{\zeta_{i}}-\overline{z_{i}}\right)d\left(\zeta_{i}-z_{i}\right)\]
and\[
Q(z,\zeta)=\frac{1}{K_{0}\rho(\zeta)}\sum_{i=1}^{n}Q_{i}(z,\zeta)d\left(\zeta_{i}-z_{i}\right),\]
where $K_{0}$ is a large constant chosen so that \begin{equation}
\Re\mathrm{e}\left(\rho(\zeta)+\frac{1}{K_{0}}S(z,\zeta)\right)<\frac{\rho(\zeta)}{2}.\label{eq|real-part-Q-zeta-z-plus-1}\end{equation}
Notice that $\Re\mathrm{e}S\leq\chi\Re\mathrm{e}S_{0}\leq-C\rho(\zeta)$,
by \prettyref{eq|local-estimate-supp-func}, so it suffices to take
$K_{0}\geq2C$ . Remark also that, if $\zeta\in\partial\Omega$, \eqref{eq|local-estimate-supp-func}
implies $\Re\mathrm{e}S(z,\zeta)<0$ for $z\in\Omega$.

We point out also that $Q$ is not holomorphic in $z$ and that $s$
satisfies\[
\left|z-\zeta\right|^{2}=\left|\left\langle s,z-\zeta\right\rangle \right|\leq C\left|z-\zeta\right|,\, z,\zeta\in\Omega.\]

Following the construction done in \cite{BA82}, with $G(\xi)=\frac{1}{\xi}$,
we obtain two kernels\begin{equation}
K(z,\zeta)=C_{n}\sum_{k=0}^{n-1}\frac{(n-1)!}{k!}G^{(k)}\left(\frac{1}{K_{0}\rho(\zeta)}S(z,\zeta)+1\right)\frac{s(z,\zeta)\wedge\left(dQ\right)^{k}\wedge\left(ds\right)^{n-k-1}}{\left|\zeta-z\right|^{2(n-k)}}\label{eq|Koppelman-kernel-K}\end{equation}
and\begin{equation}
P(z,\zeta)=C{}_{n}'G^{(n)}\left(\frac{1}{K_{0}\rho(\zeta)}S(z,\zeta)+1\right)\left(dQ\right)^{n}\label{eq|Koppelman-kernel-P}\end{equation}
giving the following Koppelman formula:

\medskip{}

\emph{For all $\left(0,q\right)$-form $f$, with $\mathcal{C}^{1}(\bar{\Omega})$
coefficients, we have, for $z\in\Omega$,\begin{equation}
f(z)=\int_{\partial\Omega}f(\zeta)\wedge K^{0}(z,\zeta)+(-1)^{q+1}\overline{\partial_{z}}\int_{\Omega}f(\zeta)\wedge K^{1}(z,\zeta)+(-1)^{q}\int_{\Omega}\bar{\partial}f(\zeta)\wedge K^{2}(z,\zeta)-\int_{\Omega}f(\zeta)\wedge P(z,\zeta)\label{eq|General-Koppelman-formula}\end{equation}
where $K^{0}$ (resp. $K^{1}$, resp. $K^{2}$, resp. $P$) is the
component of $K$ of be-degree $\left(0,q\right)$ in $z$ and $\left(n,n-q-1\right)$
in $\zeta$ (resp. $\left(0,q-1\right)$ in $z$ and $\left(n,n-q\right)$
in $\zeta$, resp. $\left(0,q\right)$ in $z$ and $\left(n,n-q-1\right)$
in $\zeta$, resp. $\left(0,q\right)$ in $z$ and $\left(n,n-q\right)$
in $\zeta$).}

\medskip{}

Moreover, by definition of $S$, $G^{(k)}\left(S(z,\zeta)+1\right)=\frac{c_{k}\rho(\zeta)^{k+1}}{\left[\frac{1}{K_{0}}S(z,\zeta)+\rho(\zeta)\right]^{k+1}}$,
and, for $\zeta\in\partial\Omega$, $K^{0}(z,\zeta)=0$ so that the
first integral in the Koppelman formula disappears, and if $f$ is
$\bar{\partial}$-closed \prettyref{eq|General-Koppelman-formula}
becomes\begin{equation}
f(z)=(-1)^{q+1}\overline{\partial_{z}}\int_{\Omega}f(\zeta)\wedge K^{1}(z,\zeta)-g,\label{eq|Koppelman-formula-for-d-bar-closed}\end{equation}
with\[
g=\int_{\Omega}f(\zeta)\wedge P(z,\zeta)\]
and $g$ is $\bar{\partial}$-closed.

\medskip{}

To be able to estimate the kernels $K^{1}$ and $P$, we need a fundamental
estimate for $\left|\rho(\zeta)+\frac{1}{K_{0}}S(z,\zeta)\right|$.
\begin{sllem}
\label{lem|lemma-2.1}There exists $K_{0}$ such that, for $\zeta\in P_{\varepsilon}^{i}(z)=P_{2^{-i}\varepsilon}(z)\setminus P_{2^{-i-1}\varepsilon}(z)$,
we have:
\begin{enumerate}
\item $\left|\rho(\zeta)+\frac{1}{K_{0}}S(z,\zeta)\right|\gtrsim2^{-i}\varepsilon$,
$\left(z,\zeta\right)\in\bar{\Omega}\times\bar{\Omega}$;
\item $\left|\frac{1}{K_{0}}S(\zeta,z)\right|\gtrsim2^{-i}\varepsilon$,
$\left(z,\zeta\right)\in\partial\Omega\times\bar{\Omega}$.
\end{enumerate}
\end{sllem}
\begin{proof}
Clearly (2) is a special case of (1) (because $d(z,\zeta)\simeq d(\zeta,z)$).
Let us prove (1).

First note that, by \prettyref{eq|local-estimate-supp-func}, if $z=\zeta+w_{1}n_{\zeta}+w_{2}t$,\[
\Re\mathrm{e}\left(\rho(\zeta)+\frac{1}{K_{0}}S(z,\zeta)\right)\leq\chi(z,\zeta)\left[c_{1}\rho(z)+\frac{1}{2}\rho(\zeta)\right]-\left(1-\chi(z,\zeta)\right)\left|z,\zeta\right|^{2}-c_{2}\sum_{j=2}^{n}\left\Vert P_{\zeta,t}^{j}\right\Vert \left|w_{2}\right|^{j},\]
and then, if $\left|z-\zeta\right|>\varepsilon_{0}$, $\Re\mathrm{e}\left(\rho(\zeta)+\frac{1}{K_{0}}S(z,\zeta)\right)\lesssim-c_{0}\left(\varepsilon_{0}\right)$,
and it is enough to prove the Lemma for $\varepsilon<\varepsilon_{0}$
small enough.

Let us denote $\varepsilon'=2^{-i-1}\varepsilon$ and let $\left(w_{1},\ldots,w_{n}\right)$
be an $\varepsilon'$-extremal basis at $\zeta$. Let us write $z=\zeta+\sum_{i=1}^{n}\lambda_{i}w_{i}=\zeta+\lambda_{1}w_{1}+v_{1}=\zeta+\lambda_{1}w_{1}+\left\Vert v_{1}\right\Vert v$.
Remark that, by \eqref{eq|real-part-Q-zeta-z-plus-1}, the result
is trivial if $\left|\rho(\zeta)\right|\gtrsim\varepsilon'$. Let
us suppose $\left|\rho(\zeta)\right|\ll\varepsilon'$.

Let $0<k_{0}\ll1$ be a real number which will be fixed later.

Suppose first that, $\forall i\geq2$, $\left|\lambda_{i}\right|<k_{0}\tau_{i}(\zeta,\varepsilon')$.
Then $\left|\lambda_{1}\right|\gtrsim\varepsilon'$ (property \eqref{geometry-4}
of \prettyref{sec|Geometry-of-lineally-convex}), and, as\[
\sum_{j=2}^{m}M^{2^{j}}\sigma_{j}\sum_{\SU{\left|\alpha\right|=j\AS\alpha_{1}=0}}\frac{\partial\rho_{\zeta}(0)}{\partial\xi^{\alpha}}\xi^{\alpha}=\sum_{j=2}^{m}M^{2^{j}}\sigma_{j}\sum_{\SU{\left|\alpha\right|=j\AS\alpha_{1}=0}}\frac{\partial\rho(z)}{\partial w^{\alpha}}_{|w=0}(\lambda w)^{\alpha},\]
by \eqref{geometry-1} of \prettyref{sec|Geometry-of-lineally-convex},
we have\[
\left|\sum_{j=2}^{m}M^{2^{j}}\sigma_{j}\sum_{\SU{\left|\alpha\right|=j\AS\alpha_{1}=0}}\frac{\partial\rho(z)}{\partial w^{\alpha}}_{|w=0}(\lambda w)^{\alpha}\right|\lesssim k_{0}^{2}\varepsilon'.\]

Using formula \prettyref{eq|definition-formula-supp-func}, this gives\begin{eqnarray*}
\left|\rho(\zeta)+\frac{1}{K_{0}}S_{0}(z,\zeta)\right| & \geq & \left|\rho(\zeta)+\frac{\lambda_{1}\left(1-A_{\zeta}(z)\right)}{K_{0}}+\frac{K\lambda_{1}^{2}\left(1-A_{\zeta}(z)\right)^{2}}{K_{0}}\right|-C_{1}k_{0}^{2}\varepsilon'\\
 & \geq & \left|\rho(\zeta)+\frac{\lambda_{1}\left(1-A_{\zeta}(z)\right)}{K_{0}}\right|-C_{2}\varepsilon'\left(k_{0}^{2}+\varepsilon'\right),\end{eqnarray*}
the last inequality following the fact that $\left|\lambda_{1}\right|\lesssim\varepsilon'$
(property \eqref{geometry-4} of \prettyref{sec|Geometry-of-lineally-convex}).
As $\left|A_{\zeta}(z)\right|<10^{-1}$, it follows that\[
\left|\rho(\zeta)+\frac{1}{K_{0}}S_{0}(z,\zeta)\right|\gtrsim\varepsilon',\]
if $k_{0}$ and $\varepsilon_{0}$ are chosen sufficiently small.

We now fix $\varepsilon_{0}$. As $\Re\mathrm{e}\left(\rho(\zeta)+\frac{1}{K_{0}}S_{0}(z,\zeta)\right)<0$,
we have\[
\left|\rho(\zeta)+\frac{1}{K_{0}}S(z,\zeta)\right|\gtrsim\chi(z,\zeta)\left|\rho(\zeta)+\frac{1}{K_{0}}S_{0}(z,\zeta)\right|+\left(1-\chi(z,\zeta)\right)\left|z-\zeta\right|^{2},\]
and the inequality $\left|z-\zeta\right|^{2}\gtrsim\left\Vert v_{1}\right\Vert ^{2}\gtrsim\tau(\zeta,v,\varepsilon)^{2}\gtrsim\varepsilon'$
(properties \eqref{geometry-3} and \eqref{geometry-5} of \prettyref{sec|Geometry-of-lineally-convex})
proves the Lemma in that case.

Suppose then there exists $i\geq2$ such that $\left|\lambda_{i}\right|\geq k_{0}\tau_{i}(\zeta,\varepsilon')$.
By (\ref{geometry-4}) of \prettyref{sec|Geometry-of-lineally-convex},
we have $\left\Vert v_{1}\right\Vert \simeq\tau(\zeta,v,\varepsilon')$,
and by (\ref{geometry-2}) of \prettyref{sec|Geometry-of-lineally-convex},
we get ($v$ being tangent to $\rho=\rho(\zeta)$ at $\zeta$)\[
\sum_{2\leq\left|\alpha+\beta\right|\leq m}\left|a_{\alpha\beta}^{v}(\zeta)\right|\tau(\zeta,v,\varepsilon)^{\alpha+\beta}=\sum_{1\leq\left|\alpha+\beta\right|\leq m}\left|a_{\alpha\beta}^{v}(\zeta)\right|\tau(\zeta,v,\varepsilon)^{\alpha+\beta}\simeq\varepsilon',\]
and \prettyref{eq|local-estimate-supp-func} implies\[
\Re\mathrm{e}S_{0}(z,\zeta)\lesssim C\rho(\zeta)-c_{1}\varepsilon',\]
 and we conclude by the same argument on $\left|z-\zeta\right|^{2}$.\end{proof}
\begin{sllem}\label{lem|P-C-infinity}
If $q\geq1$, all the derivatives, in $z$, of $P(z,\zeta)$ are uniformly
bounded in $\bar{\Omega}\times\bar{\Omega}$.\end{sllem}
\begin{proof}
Recall $P(z,\zeta)=c_{n}\frac{\rho(\zeta)^{n+1}}{\left(\frac{1}{K_{0}}S(z,\zeta)+\rho(\zeta)\right)^{n+1}}\left(dQ\right)^{n}$.
If $\delta_{\partial\Omega}(\zeta)>\nicefrac{\eta_{0}}{2}$, this
is clear by \prettyref{eq|real-part-Q-zeta-z-plus-1}. Suppose $\delta_{\partial\Omega}(\zeta)\leq\nicefrac{\eta_{0}}{2}$.
If $\left|z-\zeta\right|<\nicefrac{R}{2}$ then $Q$ is holomorphic
in $z$ and the component of $\left(dQ\right)^{n}$ of be-degree $\left(0,q\right)$
in $z$ (recall we suppose $q\geq1$) is identically $0$ which implies
$P=0$. If $\left|z-\zeta\right|\geq\nicefrac{R}{2}$, the preceding
Lemmas, give $\left|\frac{1}{K_{0}}S(z,\zeta)+\rho(\zeta)\right|\gtrsim\left(R\right)^{2m}$
(because, for unitary $v$, $\tau(z,v,\varepsilon)\lesssim\varepsilon^{\nicefrac{1}{2m}}$),
and the Lemma follows easily.\end{proof}
\begin{sllem}\label{lem|Sob-Est-P}
For each positive integer $s$ there exists a linear operator $\mathcal{L}_s$ and a constant $C(s)$ such
that, if $g=\int_{\Omega}f(\zeta)\wedge P(z,\zeta)$ is the $\bar{\partial}$-closed
form of \prettyref{eq|Koppelman-formula-for-d-bar-closed} there exists
a solution $v_{s}=\mathcal{L}_s(g)$ to the equation $\bar{\partial}v_{s}=g$ satisfying
the estimate\[
\left\Vert v_{s}\right\Vert _{s}\leq C(s)\left\Vert f\right\Vert _{L^{1}(\Omega)},\]
where $\left\Vert .\right\Vert _{s}$ denotes the Sobolev norm of
index $s$ in $\Omega$.\end{sllem}
\begin{proof}
This is an immediate consequence of the preceding Lemma and the results
on the weighted $\bar{\partial}$-Neumann problem obtained by J. J.
Kohn in \cite{Kohn-defining-function}.
\end{proof}

As explained at the end of the introduction of \prettyref{sec|resolution-d-bar-equation}
this last Lemma shows that to obtain sharp estimates for solution
of the $\bar{\partial}$-equation in lineally convex domains of finite
type, like the one stated in \prettyref{thm|L1-estimate_boundary_d-bar},
it is enough to prove that the form $z\mapsto\int_{\Omega}f(\zeta)\wedge K^{1}(z,\zeta)$
(c.f.\prettyref{eq|Koppelman-formula-for-d-bar-closed}) is continuous on $\overline\Omega$
and that
\[\left\Vert\int_{\Omega}f(\zeta)\wedge K^{1}(z,\zeta)\right\Vert_{L^1(\partial\Omega)}\leq C|\!|\!|f|\!|\!|_k.\]
This will be done in the next section.

\subsection{Proof of \prettyref{thm|L1-estimate_boundary_d-bar}}

By formulas \prettyref{eq|Koppelman-kernel-K} and \prettyref{eq|General-Koppelman-formula},
we have\[
K^{1}(z,\zeta)=\sum_{k=n-q}^{n-1}C'_{k}\frac{\rho(\zeta)^{k+1}s\wedge\left(\partial_{\bar{\zeta}}Q\right)^{n-q}\wedge\left(\partial_{\bar{z}}Q\right)^{k+q-n}\wedge\left(\partial_{\bar{z}}s\right)^{n-k-1}}{\left|z-\zeta\right|^{2\left(n-k\right)}\left(\frac{1}{K_{0}}S(z,\zeta)+\rho(\zeta)\right)^{k+1}}.\]

A priory, formula \prettyref{eq|General-Koppelman-formula} is only
valid for $z\in\Omega$. But, as noted by H. Skoda in \cite{Sko76},
the form\[
\int_{\Omega}f(\zeta)\wedge K^{1}(z,\zeta)\]
is continuous in $\bar{\Omega}$ if the kernel $K^{1}$ satisfies
a condition of uniform integrability:\begin{equation}
\int_{\Omega\cap P_{\varepsilon}(z)}\left|K^{1}(z,\zeta)\right|dV(\zeta)=\mathrm{O}\left(\varepsilon^{\nicefrac{1}{m}}\right),\label{eq|uniform-integrability-condition}\end{equation}
uniformly for $z$ satisfying $\delta_{\partial\Omega}(z)\leq\varepsilon$
and $\varepsilon$ small enough.

Under these hypothesis ($\zeta\in P_{\varepsilon}(z)$), $Q(z,\zeta)=\frac{1}{K_{0}\rho(\zeta)}\sum_{i=1}^{n}Q_{i}(z,\zeta)d\left(\zeta_{i}-z_{i}\right)$
is holomorphic in $z$ and $K^{1}$ is reduced to\[
K^{1}(z,\zeta)=c\frac{\rho(\zeta)^{n-q+1}s\wedge\left(\partial_{\bar{\zeta}}Q\right)^{n-q}\wedge\left(\partial_{\bar{z}}s\right)^{q-1}}{\left|z-\zeta\right|^{2q}\left(\frac{1}{K_{0}}S(z,\zeta)+\rho(\zeta)\right)^{n-q+1}}.\]

To prove \prettyref{eq|uniform-integrability-condition}, we use the
coordinate system $\left(\zeta_{1},\ldots,\zeta_{n}\right)$ associated
to the $\left(z,\delta_{\partial\Omega}(z)\right)$-extremal basis
and the following estimates:
\begin{sllem}
\label{lem|Lemma-2.5}For $z$ close to $\partial\Omega$, $\varepsilon$
small and $\zeta\in P_{\varepsilon}(z)$, we have:
\begin{enumerate}
\item $\left|\frac{\partial\rho}{\partial\zeta_{i}}(\zeta)\right|\lesssim\frac{\varepsilon}{\tau_{i}(z,\varepsilon)}$
(\eqref{geometry-1} of \prettyref{sec|Geometry-of-lineally-convex});
\item $\left|Q_{i}(z,\zeta)\right|+\left|Q_{i}(\zeta,z)\right|\lesssim\frac{\varepsilon}{\tau_{i}(z,\varepsilon)}$
(see \cite{Diederich-Fischer_Holder-linally-convex});
\item $\left|\frac{\partial Q_{i}(z,\zeta)}{\partial\overline{\zeta_{j}}}\lesssim\frac{\varepsilon}{\tau_{i}(z,\varepsilon)\tau_{j}(z,\varepsilon)}\right|$
(see \cite{Diederich-Fischer_Holder-linally-convex}).
\end{enumerate}
\end{sllem}

A straightforward calculus shows that\begin{eqnarray*}
\left(\partial_{\bar{\zeta}}Q\right)^{n-q} & = & \left(\frac{1}{\rho(\zeta)}\right)^{n-q}\sideset{}{'}\sum_{\SU{I=\left(i_{1},\ldots,i_{n-q}\right)\AS J=\left(j_{1},\ldots,j_{n-q}\right)}}\prod_{k=1}^{n-q}\frac{\partial Q_{i_{k}}(z,\zeta)}{\partial\overline{\zeta_{j_{k}}}}\bigwedge_{i\in I}d\zeta_{i}\bigwedge_{j\in J}d\overline{\zeta_{j}}\pm\\
 &  & \pm\left(\frac{1}{\rho(\zeta)}\right)^{n-q+1}\sideset{}{'}\sum_{\SU{I=\left(i_{1},\ldots,i_{n-q}\right)\AS J=\left(j_{1},\ldots,j_{n-q}\right)}}\sum_{k_{0}\in\left\{ 1,\ldots,n-q\right\} }\frac{\partial\rho(\zeta)}{\partial\overline{\zeta_{j_{k_{0}}}}}Q_{i_{k_{0}}}(z,\zeta)\prod_{\SU{1\leq k\leq n-q\AS k\neq k_{0}}}\frac{\partial Q_{i_{k}}(z,\zeta)}{\partial\overline{\zeta_{j_{k}}}}\bigwedge_{i\in I}d\zeta_{i}\bigwedge_{j\in J}d\overline{\zeta_{j}},\end{eqnarray*}
where $\sideset{}{'}\sum$ means that the $i_{r}$ (resp. $j_{r}$) are all different.

Suppose $\zeta\in P_{\varepsilon}^{0}=P_{\varepsilon}(z)\setminus P_{\nicefrac{\varepsilon}{2}}(z)$.
Then, by \prettyref{lem|lemma-2.1}, $\left|z-\zeta\right|^{2q}\left(\frac{1}{K_{0}}S(z,\zeta)+\rho(\zeta)\right)^{n-q+1}\gtrsim\varepsilon^{n-q+1}\left|z-\zeta\right|^{2q}$,
and, by \prettyref{lem|Lemma-2.5},\begin{eqnarray*}
\left|K^{1}(z,\zeta)\right| & \lesssim & \sideset{}{'}\sum_{\SU{I=\left(i_{1},\ldots,i_{n-q}\right)\AS J=\left(j_{1},\ldots,j_{n-q}\right)}}\frac{\left(\rho(\zeta)+\varepsilon\right)\varepsilon^{n-q}}{\prod_{k=1}^{n-q}\tau_{i_{k}}\tau_{j_{k}}\varepsilon^{n-q+1}\left|z-\zeta\right|^{2q-1}}\\
 & \lesssim & \sideset{}{'}\sum_{\SU{I=\left(i_{1},\ldots,i_{n-q}\right)\AS J=\left(j_{1},\ldots,j_{n-q}\right)}}\frac{1}{\prod_{k=1}^{n-q}\tau_{i_{k}}\tau_{j_{k}}\left|z-\zeta\right|^{2q-1}},\end{eqnarray*}
because $\delta_{\partial\Omega}(z)<\varepsilon$ and using \eqref{geometry-1}
of \prettyref{sec|Geometry-of-lineally-convex}.

The $\tau_{i}$ being supposed ordered increasingly, it is easy to
see that\[
\int_{P_{\varepsilon}^{0}}\frac{1}{\left|z-\zeta\right|^{2q-1}}\lesssim\int_{P_{\varepsilon}^{0}}\frac{1}{\left|z-\zeta\right|^{2q-1}+\varepsilon^{2q-1}}\lesssim\tau_{n-q+1}(z,\varepsilon)\prod_{i=1}^{n-q}\tau_{i}^{2}(z,\varepsilon),\]
and we obtain\[
\int_{P_{\varepsilon}^{0}}\left|K^{1}(z,\zeta)\right|\, dV(\zeta)\lesssim\tau_{n-q+1}(z,\varepsilon)=\mathrm{O}\left(\varepsilon^{\nicefrac{1}{m}}\right),\]
and, if we denote $P_{\varepsilon}^{i}=P_{\varepsilon}^{i}(z)=P\left(z,2^{-i}\varepsilon\right)\setminus P\left(z,2^{-(i+1)}\varepsilon\right)$,\[
\int_{P_{\varepsilon}}\left|K^{1}(z,\zeta)\right|\, dV(\zeta)=\sum_{i=0}^{\infty}\int_{P_{\varepsilon}^{i}}\left|K^{1}(z,\zeta)\right|\, dV(\zeta)\lesssim\sum_{i=0}^{\infty}\mathrm{O}\left(\left(\varepsilon2^{-i}\right)^{\nicefrac{1}{m}}\right)=\mathrm{O}\left(\varepsilon^{\nicefrac{1}{m}}\right),\]
which ends the proof of \prettyref{eq|uniform-integrability-condition}.

\medskip{}

To finish the proof of \prettyref{thm|L1-estimate_boundary_d-bar},
by Fubini's Theorem, we have to prove that\[
\int_{\partial\Omega}\left|K^{1}(z,\zeta)\wedge f(\zeta)\right|\, d\sigma(z)\lesssim\left\Vert f(\zeta)\right\Vert _{k},\]
and, by \prettyref{lem|lemma-2.1}, it is enough to see that, for
$\zeta$ near the boundary and $\eta$ small enough (to have the reduced
form of $K^{1}$),\[
\int_{\partial\Omega\cap P_{\eta}(\zeta)}\left|K^{1}(z,\zeta)\wedge f(\zeta)\right|\, d\sigma(z)\lesssim\left\Vert f(\zeta)\right\Vert _{k}.\]

To see it, let us denote $\varepsilon=\delta_{\partial\Omega}(\zeta)$,
$Q_{\varepsilon}^{0}(\zeta)=P_{\varepsilon}(\zeta)$ and $Q_{\varepsilon}^{i}(\zeta)=P_{2^{i}\varepsilon}(\zeta)\setminus P_{2^{i-1}\varepsilon}(\zeta)$,
$i\geq1$, and let us estimate\[
\int_{\partial\Omega\cap Q_{\varepsilon}^{i}(\zeta)}\left|K^{1}(z,\zeta)\wedge f(\zeta)\right|\, d\sigma(z)\]
using the coordinate system associated to the $\left(\zeta,2^{i}\varepsilon\right)$-extremal
basis. By the properties of the norm $\left\Vert .\right\Vert _{k}$
it is enough to prove the estimate when the form $f$ is (in the extremal
coordinate system) $f=\left(d\overline{\zeta}\right)^{I}$, with $\left|I\right|=q$.
Then, denoting $I=\left(i_{1},\ldots,i_{q}\right)$ and $J=\left(j_{1},\ldots,j_{n-q}\right)$,
$I\cup J=\left\{ 1,\ldots,n\right\} $, $K^{1}(z,\zeta)\wedge f(\zeta)$
is a sum of expressions of the form $\frac{W_{i}}{D}$, $i=1,2$,
with \[
D(\zeta,z)=\left|z-\zeta\right|^{2q}\left(\frac{1}{K_{0}}S(z,\zeta)+\rho(\zeta)\right)^{n-q+1}\]
and\[
W_{1}=\left(\zeta_{m}-z_{m}\right)\rho(\alpha)\prod_{k=1}^{n-q}\frac{\partial Q_{i_{k}}(z,\zeta)}{\partial\overline{\zeta_{j_{k}}}}\bigwedge_{i=1}^{n}\left(d\zeta_{i}\wedge d\overline{\zeta_{i}}\right)\bigwedge_{\SU{l\in L\AS\left|L\right|=q-1}}d\overline{z_{l}}\]
and, if $1\notin I$,\[
W_{2}=\left(\zeta_{m}-z_{m}\right)\frac{\partial\rho(\zeta)}{\partial\overline{\zeta_{1}}}Q_{i_{k_{0}}}(\zeta,z)\prod_{\SU{1\leq k\leq n-q\AS k\neq k_{0}}}\frac{\partial Q_{i_{k}}(z,\zeta)}{\partial\overline{\zeta_{j_{k}}}}\bigwedge_{i=1}^{n}\left(d\zeta_{i}\wedge d\overline{\zeta_{i}}\right)\bigwedge_{\SU{l\in L\AS\left|L\right|=q-1}}d\overline{z_{l}}.\]

Now, remarking that\[
\int_{Q_{\varepsilon}^{i}(\zeta)\cap\partial\Omega}\frac{d\sigma(z)}{\left|\zeta-z\right|^{2q-1}}\lesssim2^{i}\varepsilon\tau_{n-q+1}\left(\zeta,2^{i}\varepsilon\right)\prod_{i=2}^{n-q}\tau_{i}^{2}\left(\zeta,2^{i}\varepsilon\right),\]
we get, using property (\ref{geometry-5}) of \prettyref{sec|Geometry-of-lineally-convex},\[
\int_{Q_{\varepsilon}^{i}(\zeta)\cap\partial\Omega}\left|\frac{W_{1}(\zeta,z)}{D(\zeta,z)}\wedge f(\zeta)\right|\, d\sigma(z)\lesssim\frac{\rho(\zeta)}{2^{i}\varepsilon}\min_{j\in I}\frac{\tau_{j}\left(\zeta,2^{i}\varepsilon\right)}{2^{i}\varepsilon}\lesssim\frac{\rho(\zeta)}{2^{i}\varepsilon}\min_{j\in I}\frac{\tau_{j}(\zeta,\varepsilon)}{\varepsilon}\lesssim\min_{j\in I}\frac{\tau_{j}\left(\zeta,\delta_{\partial\Omega}(\zeta)\right)}{2^{i}\delta_{\partial\Omega}(\zeta)}\]
and, because $1\notin I$,\[
\int_{Q_{\varepsilon}^{i}(\zeta)\cap\partial\Omega}\left|\frac{W_{2}(\zeta,z)}{D(\zeta,z)}\wedge f(\zeta)\right|\, d\sigma(z)\lesssim\min_{j\in I}\frac{\tau_{j}\left(\zeta,2^{i}\varepsilon\right)}{2^{i}\varepsilon}\lesssim\min_{j\in I}\frac{\tau_{j}(\zeta,\delta_{\partial\Omega}(\zeta))}{2^{\nicefrac{i}{2}}\delta_{\partial\Omega}(\zeta)},\]
and the proof is complete.

\section{\label{sec|Nonisotropic-estimates-of-positive-currents}Nonisotropic
estimates of closed positive currents on geometrically separated domains}

In this Section, $\Omega$ is a pseudo-convex domain in $\mathbb{C}^{n}$
with $\mathcal{C}^{\infty}$ boundary which is geometrically separated
at a point $p_{0}\in\partial\Omega$. To state the main result of
this Section we fix some notations.

We denote by $\rho$ a defining function of $\Omega$ and $N=\frac{1}{\left|\nabla\rho\right|^{2}}\sum_{i}\frac{\partial\rho}{\partial\overline{z_{i}}}dz_{i}$
the unit normal defined in a neighborhood $U$ of $\partial\Omega$.
Recall (see \cite{CD08}) that the hypothesis made on $\Omega$ means
that there exists a constant $K>0$, a neighborhood $V=V(p_{0})\subset U$
of $p_{0}$ and a $(n-1)$-dimensional complex vector space $E^{0}$
of $(1,0)$-vector fields $\mathcal{C}^{\infty}$ in $V$, tangent
to $\rho$ (i.e. $L(\rho)\equiv0$ in $V$ for $L\in E^{0}$) such
that, at every point $p$ of $V\cap\overline{\Omega}$ and for every
$0<\delta<\delta_{0}$ there exists a $\left(K,p,\delta)\right)$-extremal
(or $\left(p,\delta\right)$-extremal) basis whose elements belong
to $E^{0}$. We will denote by $E^{1}$ the complex vector space generated
by $E^{0}$ and $N$.

For $L\in E^{1}$, $0<\varepsilon<\delta_{0}$ and $z\in V\cap\overline{\Omega}$,
in \cite{CD08} we defined the weight $F(L,z,\varepsilon)$ as follows:

If $L=L_{\tau}+a_{n}N$, the weight is defined by\[
F(L,z,\varepsilon)=\sum_{\mathcal{L}\in\mathcal{L}_{M}(L)}\left|\frac{\mathcal{L}(\partial\rho)(z)}{\varepsilon}\right|^{2/\left|\mathcal{L}\right|}+\frac{\left|a_{n}\right|^{2}}{\varepsilon^{2}},\]
where $M$ is an integer larger than the type, $\mathcal{L}_{M}(L)$
denotes the set of lists $\left(L^{1},\ldots L^{k}\right)$, of length
$k\leq M$, $L^{j}\in\left\{ L_{\tau},\overline{L_{\tau}}\right\} $,
and\[
\mathcal{L}(\partial\rho)(z)=L^{1}\ldots L^{k-2}\left(\left\langle \partial\rho,\left[L^{k-1},L^{k}\right]\right\rangle \right)(z).\]

\begin{rem*}
To keep the same notations as in \cite{CD08}, the normal is denoted
by $L_{n}$ in opposite of the preceding Section where it was denoted
by $L_{1}$.
\end{rem*}
Now we denote $\tau\left(L,z,\varepsilon\right)=F\left(L,z,\varepsilon\right)^{-\nicefrac{1}{2}}$.
If $v$ is any non zero vector in $\mathbb{C}^{n}$, and $z\in V\cap\overline{\Omega}$,
there exists a unique vector field $L=L_{z}$ in $E^{1}$ such that
$L(z)=v$. Then we denote\[
\tau\left(z,v,\varepsilon\right)=\tau\left(L,z,\varepsilon\right)\mbox{ and }k\left(z,v,\varepsilon\right)=\frac{\delta_{\partial\Omega}(z)}{\tau\left(z,v,\varepsilon\right)},\]
 where $\delta_{\partial\Omega}(z)$ is the distance from $z$ to
the boundary of $\Omega$. When $\varepsilon=\delta_{\partial\Omega}(z)$
we denote $\tau\left(z,v\right)=\tau\left(z,v,\delta_{\partial\Omega}(z)\right)$
and $k\left(z,v\right)=k\left(z,v,\delta_{\partial\Omega}(z)\right)$.

To simplify the exposition, we will also use the following terminology.
For each point $z\in V\cap\overline{\Omega}$ and each $0<\delta\leq\delta_{0}$,
if $\left(v_{i}\right)_{1\leq i\leq n}$ is a basis of $\mathbb{C}^{n}$
such that $v_{n}=N(z)$ and $\left(L_{i}\right)_{1\leq i\leq n-1}$
($L_{i}\in E^{0}$) is a $\left(z,\delta\right)$-extremal basis such
that $L_{i}(z)=v_{i}$ we will say that $\left(v_{i}\right)_{1\leq i\leq n}$
is a $\left(z,\delta\right)$-extremal basis.

Suppose now that $\Theta$ is a $\mathcal{C}^{\infty}$ $\left(1,1\right)$-form
defined in $V\cap\overline{\Omega}$. Then, following \cite{Bruna-Charp-Dupain-Annals},
we define, for $z\in V\cap\overline{\Omega}$, \[
\left\Vert \Theta(z)\right\Vert _{k}=\left\Vert \Theta(z)\right\Vert _{k}^{p_{0}}=\sup_{u,v\in\mathbb{C}^{*}}\frac{\left|\Theta(z)\left(u,v\right)\right|}{k(z,u)k(z,v)}\]
 and\[
\left\Vert \Theta(z)\right\Vert _{E}=\sup_{u,v\in\mathbb{C}^{*}}\frac{\left|\Theta(z)\left(u,v\right)\right|}{\left\Vert u\right\Vert \left\Vert v\right\Vert },\]
 where $\left\Vert .\right\Vert $ denotes the euclidean norm.

Similarly, with the above notations, we extend the notations defined
in formulas \prettyref{eq|anisotropic_distance_for_0-q_forms} and
\prettyref{eq|k-morm-measure} for $\left(0,q\right)$-forms to forms
defined in a geometrically separated domain.

The aim of the Section is to prove the following result:
\begin{stthm}
\label{thm|theorem3.1}Let $\Omega$ be a pseudo-convex domain in
$\mathbb{C}^{n}$ which is geometrically separated at a boundary point
$p_{0}$. Then there exist two neighborhoods $V$ and $W\subset V$
of $p_{0}$ and a constant $C>0$ such that, if $\Theta$ is a smooth
positive closed $\left(1,1\right)$-form defined in $V$ then\[
\int_{W}\delta_{\partial\Omega}(z)\left\Vert \Theta(z)\right\Vert _{k}dV(z)\leq C\int_{V}\delta_{\partial\Omega}(z)\left\Vert \Theta(z)\right\Vert _{E}dV(z).\]

\end{stthm}
If $\Omega$ is a bounded pseudo-convex domain which is geometrically
separated at every point of it's boundary, we choose a finite number
of points $p_{i}$, $1\leq i\leq N$, in $\partial\Omega$ such that
the set of the neighborhoods $V_{i}=V\left(p_{i}\right)$ is a covering
$\mathcal{V}$of $\partial\Omega$, $V_{0}=\Omega\setminus\bigcup_{i\geq1}V_{i}$,
and we define\[
\left\Vert \Theta(z)\right\Vert _{k}^{\mathcal{V}}=\max_{z\in V_{i}}\left\Vert \Theta(z)\right\Vert _{k}^{i},\]
where, $\left\Vert \Theta(z)\right\Vert _{k}^{i}=\left\Vert \Theta(z)\right\Vert _{k}^{p_{i}}$
if $i\geq1$ and $\left\Vert \Theta(z)\right\Vert _{k}^{0}=\left\Vert \Theta(z)\right\Vert _{E}$.

With this notation, the local result immediately implies a global
one:
\begin{stthm}
\label{thm|theorem-global-estimate-currents}If $\Omega$ is a pseudo-convex
domain in $\mathbb{C}^{n}$ which is geometrically separated at every
point of it's boundary, a covering $\mathcal{V}$ of $\partial\Omega$
being chosen, there exists a constant $C>0$ such that\[
\int_{\Omega}\delta_{\partial\Omega}(z)\left\Vert \Theta(z)\right\Vert _{k}^{\mathcal{V}}dV(z)\leq C\int_{\Omega}\delta_{\partial\Omega}(z)\left\Vert \Theta(z)\right\Vert _{E}dV(z)\]
 for all smooth closed positive $\left(1,1\right)$-form $\Theta$
in $\Omega$.
\end{stthm}
In fact the two statements are equivalent. In \cite{CD08} it is proved
that if $\Omega$ is geometrically separated at $p_{0}\in\partial\Omega$
then there exists a bounded pseudo-convex domain $D$ with $\mathcal{C}^{\infty}$
boundary contained in $\Omega$ which is geometrically separated at
every point of it's boundary and whose boundary contains a neighborhood
of $p_{0}$ in the boundary of $\Omega$. Thus, \prettyref{thm|theorem3.1}
for $\Omega$ follows immediately \prettyref{thm|theorem-global-estimate-currents}.

\medskip{}

We now prove the Theorems. As it is very similar to the proof of the
same result for convex domains given in Section 2 of \cite{Bruna-Charp-Dupain-Annals},
we will refer to that paper for many details. Precisely, we will only
give the mains articulations and the proof of Lemmas where the differences
due to the fact that the weights $F$ are defined with vector fields
instead of coordinates systems are relevant.

We still use the previous notations for the defining function $\rho$
of $\Omega$, the point $p_{0}\in\partial\Omega$, the neighborhood
$V$ of $p_{0}$ and $W$ a neighborhood of $p_{0}$ relatively compact
in $V$.

In Section 3.3 of \cite{CD08} it is shown that there exists a $\left(z,\varepsilon\right)$-adapted
coordinate system $\left(\xi_{i}\right)_{i}$ used, in particular,
to define {}``polydiscs'' centered at the point $z$ by\[
P_{\varepsilon}(z)=P(z,\varepsilon)=\left\{ q=\left(\xi_{i}\right)_{i}\mbox{ such that }\left|\xi_{i}\right|\leq c\tau(z,v_{i},\varepsilon)\right\} ,\]
where $v_{i}=L_{i}(z)$, $1\leq i\leq n$, $\left(L_{i}\right)_{1\leq i\leq n}$
being the $\left(z,\varepsilon\right)$-extremal basis, and $c$ a
sufficiently small constant (depending on $\Omega$). Moreover, the
set of these polydiscs are associated to a pseudo-distance. With these
notations we have
\begin{sllem}
\label{lem|lemma-1.2}For $z\in V\cap\overline{\Omega}$ and $v\in\mathbb{C}^{n}$,
$v\neq0$, if $w\in P_{\varepsilon}(z)$ we have $\tau(w,v,\varepsilon)\simeq\tau(z,v,\varepsilon)$.\end{sllem}
\begin{proof}
Let $L_{z}$ (resp. $L_{w}$) the vector field belonging to $E_{1}$
such that $L_{z}(z)=v$ (resp. $L_{w}(w)=v$). Let $\left(L_{i}\right)_{1\leq i\leq n}$
be the $\left(z,\varepsilon\right)$-extremal basis and let us write
$L_{z}=\sum_{i=1}^{n}b_{i}L_{i}$ and $L_{w}=\sum_{i=1}^{n}b_{i}'L_{i}$,
$b_{i},\, b_{i}'\in\mathbb{C}$. If $\left(\xi_{i}\right)_{1\leq i\leq n}$
is the $\left(z,\varepsilon\right)$-adapted coordinate system let
us write $L_{j}(.)=\sum_{i}a_{j}^{i}(.)\frac{\partial}{\partial\xi_{i}}$,
with $a_{j}^{i}(z)=\delta_{i}^{j}$ so that\[
L_{z}(z)=\sum_{i=1}^{n}b_{i}\frac{\partial}{\partial\xi_{i}}\mbox{ and }L_{w}(w)=\sum_{i=1}^{n}\left(\sum_{j=1}^{n}b_{j}'a_{j}^{i}(w)\right)\frac{\partial}{\partial\xi_{i}}.\]

In \cite{CD08} we proved (Proposition 3.6) that,if $w\in P_{\varepsilon}(z)$,\[
\left|a_{j}^{i}(w)\right|\lesssim F_{j}^{\nicefrac{1}{2}}(z,\varepsilon)F_{i}^{-\nicefrac{1}{2}}(z,\varepsilon).\]

The basis $\left(L_{i}\right)$ being $\left(z,\varepsilon\right)$-extremal
we have immediately that $F(L_{z},z,\varepsilon)\simeq F(L_{w},z,\varepsilon)$,
and, using another time Proposition 3.6 of \cite{CD08}, that $F(L_{w},z,\varepsilon)\simeq F(L_{w},w,\varepsilon)$
finishing the proof of the Lemma.
\end{proof}

The proof of the Theorems is done in three steps.

\medskip{}

\emph{First step: definition of a family of polydiscs.} As the polydiscs
$P_{\varepsilon}(z)$ are associated to a pseudodistance, there exists
a constant $M$ such that, for $\delta$ and $\varepsilon$ sufficiently
small, there exists points $z_{i}$, $1\leq i\leq n(\delta,\varepsilon)$
belonging to the set $\left\{ \rho=-\delta\right\} $ such that, if
$S_{\delta,\varepsilon}(z)=S_{\varepsilon}(z)=P_{\varepsilon}(z)\cap\left\{ \rho=-\delta\right\} $
and $S_{\delta,\varepsilon}^{*}(z)=S_{\varepsilon}^{*}(z)=2P_{\varepsilon}(z)\cap\left\{ \rho=-\delta\right\} $
then
\begin{enumerate}
\item $\left\{ S_{\varepsilon}(z_{i})\right\} _{i}$ is a covering of $W\cap\left\{ \rho=-\delta\right\} $;
\item for all $i$, $S_{\varepsilon}^{*}(z_{i})$ is contained in $V\cap\left\{ \rho=-\delta\right\} $;
\item for all $z\in\left\{ \rho=-\delta\right\} $, there exist at most
$M$ index $i$ such that $r\in S_{\varepsilon}^{*}(z_{i})$.
\end{enumerate}

Let $\alpha\in\left]0,1\right[$ be a parameter that will be fixed
later. For $\rho=\alpha^{k}$, $k\geq k_{0}$, let $\left(z_{i}^{k}\right)_{1\leq i\leq n_{k}}$
be a family of points satisfying these properties for $\delta=\alpha^{k}$
and $\varepsilon=\nicefrac{\alpha^{k}}{2}$. The family $\bigcup_{k\geq k_{0}}\left\{ z_{i}^{k},\,1\leq i\leq n_{k}\right\} $
will be denoted by $\left(z_{i}\right)_{i}$.

For $z$ close enough to $\partial\Omega$ and $\delta>0$ sufficiently
small, let $\pi_{\delta}(z)$ denote the point where the integral
curve of $\nabla\rho$ passing through $z$ meets the set $\left\{ \rho=-\delta\right\} $.
$V$ being small enough, $\pi_{\delta}$ is $\mathcal{C}^{\infty}$
in $V$ and any finite number of derivatives of $\pi_{\delta}$ is
bounded independently of $\delta$.

For all $z_{j}$ belonging to $\left(z_{i}\right)_{i}$, if $-\rho(z_{j})=\alpha^{k}$
we set $S_{j}=S_{\nicefrac{\alpha^{k}}{2}}(z_{j})$, $S_{j}^{*}=S_{\nicefrac{\alpha^{k}}{2}}^{*}(z_{j})$,
\[
Q_{j}=\left\{ w\in\Omega\mbox{ such that }\alpha^{k}\geq-\rho(w)\geq\alpha^{k+1},\,\pi_{\delta}(w)\in S_{j}\right\} ,\]
and\[
Q_{j}^{*}=\left\{ w\in\Omega\mbox{ such that }\alpha^{k-1}\geq-\rho(w)\geq\alpha^{k+2},\,\pi_{\delta}(w)\in S_{j}^{*}\right\} .\]

We suppose $k_{0}$ chosen sufficiently large so that $\bigcup_{j}Q_{j}^{*}\subset V$.
Moreover, note that there exist at most $4M$ index $j$ such that
$z\in Q_{j}^{*}$.

\medskip{}

\emph{Second step: some estimates related to the radius $\tau_{i}$.}
\begin{sllem}
\label{lem|lemma-1.1}Let $z\in V\cap\overline{\Omega}$, $v\in\mathbb{C}^{n}$,
$v\neq0$, and $m$ an integer larger than the type of $p_{0}$.
\begin{enumerate}
\item If $\varepsilon_{1}\geq\varepsilon_{2}$, $\left(\frac{\varepsilon_{1}}{\varepsilon_{2}}\right)^{\nicefrac{1}{2}}\tau\left(z,v,\varepsilon_{2}\right)\gtrsim\tau\left(z,v,\varepsilon_{1}\right)\gtrsim\left(\frac{\varepsilon_{1}}{\varepsilon_{2}}\right)^{\nicefrac{1}{m}}\tau\left(z,v,\varepsilon_{2}\right).$
\item If $v=\sum_{i=1}^{n}w_{i}L_{i}(z)$, where $\left(L_{i},\,1\leq i\leq n\right)$
is the $\left(z,\varepsilon\right)$-extremal basis, then\[
\sum_{i=1}^{n}\left|w_{i}\right|k\left(z,v_{i},\varepsilon\right)\simeq k\left(z,v,\varepsilon\right).\]

\end{enumerate}
\end{sllem}

\begin{sllem}
If $z=\pi_{\delta}(w)$, then
\begin{enumerate}
\item if $\delta_{\partial\Omega}(w)\leq\delta_{\partial\Omega}(z)$, $\tau(z,v)\left(\frac{\delta_{\partial\Omega}(w)}{\delta_{\partial\Omega}(z)}\right)^{\nicefrac{1}{2}}\lesssim\tau(w,v)\lesssim\tau(z,v)\left(\frac{\delta_{\partial\Omega}(w)}{\delta_{\partial\Omega}(z)}\right)^{\nicefrac{1}{m}}$,
\item if $\delta_{\partial\Omega}(w)\geq\delta_{\partial\Omega}(z)$, $\tau(z,v)\left(\frac{\delta_{\partial\Omega}(w)}{\delta_{\partial\Omega}(z)}\right)^{\nicefrac{1}{m}}\lesssim\tau(w,v)\lesssim\tau(z,v)\left(\frac{\delta_{\partial\Omega}(w)}{\delta_{\partial\Omega}(z)}\right)^{\nicefrac{1}{2}}$.
\end{enumerate}
\end{sllem}
\begin{proof}
We only need to prove 1., and, by \prettyref{lem|lemma-1.1} (1),
it is enough to prove that\[
\tau\left(z,v,\nicefrac{\delta_{\partial\Omega}(z)}{2}\right)\simeq\tau\left(w,v,\nicefrac{\delta_{\partial\Omega}(z)}{2}\right).\]

This follows the fact that, if $L_{z}=\sum_{i}a_{i}L_{i}^{0}$ and
$L_{w}=\sum_{i}b_{i}L_{i}^{0}$ are so that $L_{z}(z)=L_{w}(w)=v$,
then $L_{z}=L_{w}+\mathrm{O}\left(\delta_{\partial\Omega}(z)\right)$,
and, thus $F\left(L_{z},z,\nicefrac{\delta_{\partial\Omega}(z)}{2}\right)=F\left(L_{w},z,\nicefrac{\delta_{\partial\Omega}(z)}{2}\right)+\mathrm{O}(1)$
and $F\left(L_{w},w,\nicefrac{\delta_{\partial\Omega}(z)}{2}\right)=F\left(L_{w},z,\nicefrac{\delta_{\partial\Omega}(z)}{2}\right)+\mathrm{O}(1)$.\end{proof}
\begin{sllem}
If $\left(v_{i}\right)_{1\leq i\leq n}$ is a $\left(z_{j},\nicefrac{\delta_{\partial\Omega}(z_{j})}{2}\right)$-extremal
basis then, for $w\in S_{j}^{*}$, $\sup_{u,v}\frac{\left|\Theta(w)(u,v)\right|}{k(z_{j},u)k(z_{j},v)}\simeq\sum_{l=1}^{n}\frac{\Theta(w)(v_{l},v_{l})}{k(z_{j},v_{l})^{2}}$.\end{sllem}
\begin{proof}
The proof is exactly the same as in \cite{Bruna-Charp-Dupain-Annals}.
It follows \prettyref{lem|lemma-1.1} (2) and Cauchy-Schwarz inequality.\end{proof}
\begin{sllem}
\label{lem|lemma-1.5}Let $w$ be a point in $Q_{j}^{*}$ and $\left(\tilde{w}_{i}\right)_{i}$
be the coordinates of $\pi_{-\rho(z_{j})}(w)$ in the $\left(z_{j},\nicefrac{\delta_{\partial\Omega}(z_{j})}{2}\right)$-adapted
coordinate system $\left(z_{j}^{i}\right)_{i}$. Then\[
\left|\frac{\partial\tilde{w}_{r}}{\partial z_{j}^{l}}(w)\right|\lesssim\frac{\tau(z_{j},v_{r})}{\tau(z_{j}v_{l})},\, r\leq n-1,\, l\leq n,\]
where $\left(v_{i}\right)_{1\leq i\leq n}$ is $\left(z_{j},\nicefrac{\delta_{\partial\Omega}(z_{j})}{2}\right)$-extremal,
and\[
\left|\frac{\partial\Im\mathrm{m}(\tilde{w}_{n})}{\partial z_{j}^{l}}\right|\lesssim\frac{\delta_{\partial\Omega}(z_{j})}{\tau(z_{j},v_{l})},\]
the constants depending on $\Omega$ and $\alpha$.\end{sllem}
\begin{proof}
As in \cite{Bruna-Charp-Dupain-Annals}, the proof is reduced to the
case where $\rho(w)=\rho(z_{j})$, then one consider a $\left(w,\nicefrac{\delta_{\partial\Omega}(z_{j})}{2}\right)$-extremal
basis and uses \prettyref{lem|lemma-1.2}.
\end{proof}

\begin{proof}
[Third step: proof of \prettyref{thm|local-resolution-of-d-equation} and \prettyref{thm|theorem-global-estimate-currents}]Let
$\varphi_{0}$ be a $\mathcal{C}^{\infty}$ function supported in
$\left[-2,2\right]$ and identically equal to $1$ in $\left[-1,1\right]$,
$0\leq\varphi_{0}\leq1$. Let $\Psi=\Psi_{\alpha}$ be a $\mathcal{C}^{1}$
function supported in $\left[\alpha^{2},\alpha^{-1}\right]$ equal
to $1$ on $\left[\alpha,1\right]$, $0\leq\Psi\leq1$.

With the notations of \prettyref{lem|lemma-1.5} we define, for $w\in Q_{j}^{*}$,\[
\varphi(w)=\Psi_{\alpha}\left(\frac{\delta_{\partial\Omega}(w)}{\delta_{\partial\Omega}(z_{j})}\right)\varphi_{0}\left(\frac{\left|\Im\mathrm{m}(\tilde{w}_{n})\right|}{\tau(z_{j},v_{n})}\right)\prod_{i=2}^{n}\varphi_{0}\left(\frac{\left|\tilde{w}_{i}\right|}{\tau(z_{j},v_{i})}\right).\]

Applying Stokes's formula to the form $\left(-\rho\right)^{\nicefrac{2}{m}}\varphi\Theta\wedge\eta_{l}$,
where $\eta_{l}=i^{n-1}dz_{j}^{n}\bigwedge_{\SU{r=1\AS r\neq l}}^{n-1}dz_{j}^{r}\wedge d\overline{z_{j}^{r}}$,
on $\Omega$, using the previous Lemmas and a convenient function
$\Psi_{\alpha}$ (see \cite{Bruna-Charp-Dupain-Annals,Char-Dup-Pise}),
for all $\alpha>0$, the calculus made p. 404-408 in \cite{Bruna-Charp-Dupain-Annals}
leads to the following estimate\begin{eqnarray*}
\int_{Q_{j}}\delta_{\partial\Omega}(w)\left\Vert \Theta(w)\right\Vert _{k}dV(w) & \leq & \left[\beta(\alpha)+\nicefrac{\mathrm{O}(\mathrm{diam}(Q_{j})}{\alpha}+\varepsilon C(\alpha)\right]\int_{Q_{j}^{*}}\delta_{\partial\Omega}(w)\left\Vert \Theta(w)\right\Vert _{k}dV(w)\\
 &  & +C(\varepsilon)C(\alpha)\int_{Q_{j}^{*}}\delta_{\partial\Omega}(w)\left\Vert \Theta(w)\right\Vert _{E}dV(w),\end{eqnarray*}
with $\lim_{\alpha\rightarrow0}\beta(\alpha)=0$.

To finish the proof of the Theorems choose points $p_{i}$, $0\leq i\leq N$
such that the neighborhood $V_{0}(p_{i})$ is a covering of $\partial\Omega$.
Then choosing $\alpha$ small enough so that $\beta(\alpha)MN$ is
small, then $\varepsilon$ small enough so that $\varepsilon C(\alpha)MN$
is small and finally $k_{0}$ large enough so that $\nicefrac{\mathrm{O}(\mathrm{diam}(Q_{j})}{\alpha}$
is small, we get a $\delta_{0}>0$ such that\[
\int_{W\cap\left\{ \delta_{\partial\Omega}(w)\leq\delta_{0}\right\} }\delta_{\partial\Omega}(w)\left\Vert \Theta(w)\right\Vert _{k}dV(w)\lesssim\int_{V}\delta_{\partial\Omega}(w)\left\Vert \Theta(w)\right\Vert _{E}dV(w)\]
and\[
\int_{\Omega\cap\left\{ \delta_{\partial\Omega}(w)\leq\delta_{0}\right\} }\delta_{\partial\Omega}(w)\left\Vert \Theta(w)\right\Vert _{k}dV(w)\lesssim\int_{\Omega}\delta_{\partial\Omega}(w)\left\Vert \Theta(w)\right\Vert _{E}dV(w)\]
which proves the Theorems.
\end{proof}

\section{\label{sec|Non-isotropic-estimates-for-d}Non isotropic estimates
for the $d$-operator for geometrically separated domains}
\begin{stthm}
\label{thm|local-resolution-of-d-equation}Let $\Omega$ be a pseudo-convex
domain in $\mathbb{C}^{n}$ which is geometrically separated at a
boundary point $p_{0}$. Then there exist two neighborhoods of $p_{0}$,
$W\Subset V$ such that there exists a constant $C>0$ such that,
for any smooth closed $\left(1,1\right)$-form $\Theta$ in $V$,
there exists a smooth solution $w$ of the equation $dw=\Theta$ satisfying
the following estimate\[
\int_{W}\left\Vert w\right\Vert _{k}dV\leq C\int_{W}\delta_{\partial\Omega}(z)\left\Vert \Theta(z)\right\Vert _{k}dV(z).\]

\end{stthm}
The proof of this result is quite standard and we will just adapt
the proof made in Section 3 of \cite{Bruna-Charp-Dupain-Annals}.
\begin{proof}
We choose for $V$ a neighborhood contained in the neighborhood $V(p_{0})$
defined in \prettyref{sec|Nonisotropic-estimates-of-positive-currents}
and sufficiently small so that there exists a neighborhood $W\Subset V$
of $p_{0}$ and a point $P\in V$ such that there exists an open set
$U\subset\Omega$ containing $\left(W\cap\Omega\right)\cup\left\{ P\right\} $
which is starshaped with respect to $P$. Solving the equation $dw=\Theta$
in $U$, we can assume, without loss of generality, that the support
of $\Theta$ does not contain $P$. By translation, we can also assume
that $P$ is the origin in $\mathbb{C}^{n}$.

Now, for $z\in W\cap\Omega$, $M_{t}\,:z\mapsto tz$ being the homotopy
between $0$ and $z$, Poincaré's formula\[
w=\int_{0}^{1}M_{t}^{*}\left(i_{z_{t}}(\Theta)\right)dt,\]
where $i_{z_{t}}(\Theta)$ denotes the inner contraction of $\Theta$
with the field $Z_{t}=\frac{z}{t}$ and $M_{t}^{*}$ the pull-back
operator, gives a solution $w$ of the equation $dw=\Theta$ in $W\cap\Omega$.

To finish the proof, we have to estimate this solution, and, for this,
following the calculus made in \cite{Bruna-Charp-Dupain-Annals} p.
409, we only have to verify the following Lemma:
\begin{lem*}
There exists a constant $\alpha>0$ such that, for $0\leq t\leq1$
and $v\in\mathbb{C}^{*}$,\[
\tau\left(tz,tv,\delta_{\partial\Omega}(tz)/2\right)\geq\alpha\left(\frac{\delta_{\partial\Omega}(tz)}{\delta_{\partial\Omega}(z)}\right)^{\nicefrac{1}{m}}\tau\left(z,v,\delta_{\partial\Omega}(z)/2\right).\]
\end{lem*}
\begin{proof}
[Proof of the Lemma]By definition of the polydiscs, there exists
a constant $K$ (depending only on $c$) such that $z\in P_{K\delta_{\partial\Omega}(tz)}(tz)$.
Then, by \prettyref{lem|lemma-1.2}, $\tau\left(tz,tv,K\delta_{\partial\Omega}(tz)\right)\simeq\tau\left(z,tv,K\delta_{\partial\Omega}(tz)\right)$,
and thus $\tau\left(tz,tv,\delta_{\partial\Omega}(tz)/2\right)\simeq_{K}\tau\left(z,tv,\delta_{\partial\Omega}(tz)\right)\geq\tau\left(z,v,\delta_{\partial\Omega}(tz)\right)$,
because $t\leq1$, and the Lemma follows \prettyref{lem|lemma-1.1}.
\end{proof}
\end{proof}

From \prettyref{thm|local-resolution-of-d-equation}
the following global result is easily deduced:
\begin{stthm}
\label{thm|global-resolution-d-equation}Let $\Omega$ be a bounded
pseudo-convex domain in $\mathbb{C}^{n}$ which is geometrically separated
at every point of it's boundary. There exists a constant $C>0$ such
that, for any smooth closed $\left(1,1\right)$-form $\Theta$
in $\Omega$ whose cohomology class in $H^{2}(\Omega,\mathbb{C})$
is $0$, there exists a smooth solution of the equation
$dw=\Theta$ in $\Omega$ satisfying\[|\!|\!|w|\!|\!|_{k}\leq C|\!|\!|\Theta|\!|\!|_{k}.\]
\end{stthm}
Here, the $|\!|\!|.|\!|\!|_{k}$ norm of a smooth $\left(1,1\right)$-form
$\Theta$ is defined using the covering $V_{i}$ defined before the statement of
\prettyref{thm|theorem-global-estimate-currents}:
\[|\!|\!|\Theta|\!|\!|_{k}=\sum_{i=0}^N\int_{V_i}\left(\sup_{ u,v\in\mathbb{C}^{*}}\frac{\left|\left\langle \Theta;u,v\right\rangle \right|}{k_{i}(z,u)k_{i}(z,v)}\right)\delta_{\partial\Omega}(z)dV(z),\]
with $k_{0}(z,u)=k_{0}(z,v)=1$.

\section{Characterization of the zero-sets of the functions of the Nevanlinna
class for lineally convex domains of finite type}

It is a well-known fact that if $X$ is the zero-set of a function
$f$ of the Nevanlinna class $N(\Omega)$, then it satisfies the Blaschke
condition: if $X_{k}$ are the irreducible components of $X$ and
$n_{k}$ the multiplicity of $f$ on $X_{k}$, then\[
\sum_{k}n_{k}\int_{X_{k}}\delta_{\partial\Omega}\, d\mu_{k}<+\infty,\]
where $\mu_{k}$ is the Euclidean measure on the regular part of $X_{k}$.
Classically the data $\left\{ X_{k},n_{k}\right\} $ is called a divisor.

If $\Theta$ is the $\left(1,1\right)$-positive current classically
associated to the divisor $\left\{ X_{k},n_{k}\right\} $, this condition is
\[\int_{\Omega}\delta_{\partial\Omega}d\left(\sum_{j,k} \left|\Theta_{j,k}\right|\right)<+\infty.\]
For more details we refer to \cite{Bruna-Charp-Dupain-Annals}.

To prove \prettyref{thm|zeros-Nevalinna}, $\Omega$ satisfying automatically
some topological condition, by a standard regularization process,
it is sufficient to prove that there exists a constant $C>0$ such
that if $\Theta$ is a $\left(1,1\right)$-closed positive current,
$\mathcal{C}^{\infty}$ in $\bar{\Omega}$, there exists a function
$u$ in $\bar{\Omega}$ solution of the equation $i\partial\bar{\partial}u=\Theta$
satisfying the estimate\[
\int_{\partial\Omega}\left|u\right|\, d\sigma\leq C\int_{\Omega}\delta_{\partial\Omega}\left\Vert \Theta\right\Vert _{E}.\]
For details see \cite{Bruna-Charp-Dupain-Annals} and \cite{Sko76}.

As usual, the main part of the proof is done in two steps: resolution
of the $d$-equation and, then, resolution of the $\bar{\partial}$-equation.
The second step is done in \prettyref{sec|resolution-d-bar-equation}.
For the first one we use the results of Sections \ref{sec|Nonisotropic-estimates-of-positive-currents}
and \ref{sec|Non-isotropic-estimates-for-d}. As this two Sections
are written for general geometrically separated domains and in a local
context, we give some precisions (see also \cite{Conrad_lineally_convex}).

\subsection{Non isotropic estimates on closed positive currents in lineally convex
domains}

$\Omega$ being lineally convex of finite type it is completely geometrically
separated (\cite{CD08}). On the other hand, the proof of the Corollary
at the beginning of Section 7.1 of \cite{CD08} shows that the radius
$\tau(z,v,\varepsilon)$ defined for general geometrically separated
domains are equivalents to the ones defined by M. Conrad in \cite{Conrad_lineally_convex}
(and used in \prettyref{sec|resolution-d-bar-equation})

Then, the norm $\left\Vert \Theta(z)\right\Vert _{k}^{\nu}$ used
in \prettyref{thm|theorem-global-estimate-currents} is equivalent
to the norm $\left\Vert \Theta(z)\right\Vert _{k}$ defined with the
radius $\tau_{i}$ of \cite{Conrad_lineally_convex} (and is independent
of $\nu$). Thus, \prettyref{thm|theorem-global-estimate-currents}
means that there exists a constant $C_{1}$ such that\[
\int_{\Omega}\delta_{\partial\Omega}(z)\left\Vert \Theta(z)\right\Vert _{k}dV\leq C_{1}\int_{\Omega}\delta_{\partial\Omega}(z)\left\Vert \Theta(z)\right\Vert _{E}dV.\]

\subsection{Resolution of the $d$ equation in lineally convex domains}

As $\Omega$ is lineally convex we have $H^{2}(\Omega,\mathbb{C})=0$
(see for example \cite{Conrad_lineally_convex}). Thus, by \prettyref{thm|global-resolution-d-equation}
there exists a smooth form $w$ such that $dw=\Theta$ in $\Omega$
and $\int_{\Omega}\left\Vert w(z)\right\Vert _{k}\leq C_{2}\int_{\Omega}\delta_{\partial\Omega}(z)\left\Vert \Theta(z)\right\Vert _{k}dV$.

\section{Remarks}

The method of resolution of the equation $\bar{\partial}u=f$ presented
in \prettyref{sec|resolution-d-bar-equation} can be used to obtain
various other estimates.

For example, the estimates obtained for convex domains of finite type
in \cite{Cumenge-estimates-holder,DFF99,Fis04,Hef04,Ale05,Ale06}
can be proved for lineally convex domains using our method (see also
\cite{Diederich-Fischer_Holder-linally-convex}).

\bibliographystyle{amsalpha}

\providecommand{\bysame}{\leavevmode\hbox to3em{\hrulefill}\thinspace}
\providecommand{\MR}{\relax\ifhmode\unskip\space\fi MR }
\providecommand{\MRhref}[2]{%
  \href{http://www.ams.org/mathscinet-getitem?mr=#1}{#2}
}
\providecommand{\href}[2]{#2}

\end{document}